\documentclass[11pt]{amsart}
\usepackage{amsmath,amssymb,latexsym, dsfont}
\usepackage[small]{caption}
\usepackage{graphicx,color,mathrsfs,tikz}
\usepackage{subfigure,color,bm}
\usepackage{cite}
\usepackage[colorlinks=true,urlcolor=blue,
citecolor=red,linkcolor=blue,linktocpage,pdfpagelabels,
bookmarksnumbered,bookmarksopen]{hyperref}
\usepackage[italian,english]{babel}
\usepackage{units}
\usepackage{enumitem}
\usepackage[left=2.5cm,right=2.5cm,top=2.9cm,bottom=2.9cm]{geometry}
\usepackage[hyperpageref]{backref}

\numberwithin{equation}{section}
\newtheorem{theorem}{Theorem}[section]

\newtheorem{lemma}[theorem]{Lemma}

\newtheorem{remark}[theorem]{Remark}

\theoremstyle{definition}

\renewcommand{\epsilon}{\eps}
\renewcommand{\i}{{\rm i}}

\newcommand{\C}{{\mathbb C}}
\newcommand{\N}{{\mathbb N}}
\newcommand{\R}{{\mathbb R}}

\newcommand{\mS}{{\mathbb S}}

\newcommand{\eps}{\varepsilon}

\newcommand{\1}{{\bf 1}}
\newcommand{\pnorm}[2][]{\if #1'' \left|#2\right|_p \else \left|#2\right|_{#1} \fi}

\newcommand{\loc}{{\rm loc}}

\renewcommand{\theta}{\vartheta}
\DeclareMathOperator{\supp}{supp}

\title{On anisotropic Sobolev spaces}

\author[H.-M. Nguyen]{Hoai-Minh Nguyen}
\author[M.\ Squassina]{Marco Squassina}

\address[H.-M. Nguyen]{Department of Mathematics \newline\indent
	EPFL SB CAMA \newline\indent
	Station 8 CH-1015 Lausanne, Switzerland}
\email{hoai-minh.nguyen@epfl.ch}

\address[M.\ Squassina]{Dipartimento di Matematica e Fisica \newline\indent
	Universit\`a Cattolica del Sacro Cuore \newline\indent
	Via dei Musei 41, I-25121 Brescia, Italy}
\email{marco.squassina@unicatt.it}

\thanks{The second author is member of {\em Gruppo Nazionale per l'Analisi Ma\-te\-ma\-ti\-ca, la Probabilit\`a e le loro Applicazioni} (GNAMPA) of the {\em Istituto Nazionale di Alta Matematica} (INdAM). Part  of  this  paper  was  written  during  the  
meeting  ``Nonlinear Days in Turin'' organized in Turin, Italy, in September 2017.  The hosting institution is acknowledged.}



\begin{document}

\begin{abstract}
We investigate two types of   characterizations for  anisotropic Sobolev and BV spaces. In particular, we establish anisotropic versions of the Bourgain-Brezis-Mironescu 
formula, including the magnetic case both for Sobolev and BV functions.
\end{abstract}
\maketitle


\section{Introduction and results}

\subsection{Overview}
Around 2001, J.\ Bourgain, H.\ Brezis and P.\ Mironescu, investigated (cf.\ \cite{bourg,bourg2,bre})  the asymptotic behaviour of a class on nonlocal functionals on a domain $\Omega\subset\R^N$, including those related to the norms 
of the fractional Sobolev spaces $W^{s,p}(\R^N)$ as $s\nearrow 1$. In 
the case $\Omega=\R^N$,  their later result can be formulated as follows: if $p>1$ and $u\in W^{1,p}(\R^N)$, then 
\begin{equation}
\label{basic}
\lim_{s \nearrow 1} (1-s)\iint_{\R^{2N}} \frac{|u (x) -  u(y)|^p}{|x - y|^{N+ps}} \, dx \, dy = K_{p,N} \int_{\R^N} |\nabla u|^p \, dx,
\end{equation}
where
\begin{equation}
\label{const}
K_{p,N}=\frac{1}{p}\int_{{\mathbb S}^{N-1}} |{\boldsymbol \omega}\cdot x|^pd\sigma,
\end{equation}
being ${\boldsymbol \omega}\in {\mathbb S}^{N-1}$ any fixed vector. Here and in what follows, for a vector $x \in \R^N$, 
$|x|$ denotes its Euclidean norm. 

Given a convex, symmetric subset $K\subset \R^N$ containing the origin,  
let $\|\cdot\|_K$ be the norm in $\R^N$ which admits as unit ball the set $K$, i.e., 
\begin{equation}
\label{Knorm}
\| x\|_K :  = \inf \big\{\lambda > 0: x/ \lambda  \in K \big\}. 
\end{equation}
It is rather natural to wonder what happens to formula \eqref{basic} by replacing in the singular 
kernel $|x - y|$ with its anisotropic version $\|x - y\|_K$. In 2014, M. Ludwig \cite{ludwig,ludwig-BV} proved that, for a compactly supported function $u\in W^{1,p}(\R^N)$, there holds
\begin{equation}
\label{ludw}
\lim_{s \nearrow 1} (1-s)\iint_{\R^{2N}} \frac{|u (x) -  u(y)|^p}{\|x - y\|_K^{N+ps}} \, dx \, dy = \int_{\R^N} \|\nabla u\|^p_{Z^*_pK}\, dx,
\end{equation}
Here $\| \cdot  \|_{Z^*_pK}$ is the norm associated with the convex set $Z^*_pK$ which is the polar $L_p$ moment body of $K$ (see \eqref{Lpmoment} and \eqref{Zpk}); such quantitities  were involved in recent important applications within convex geometry and probability theory, see e.g.\ \cite{fleury,Haber,ludwig-0} and the references therein. Thus, changing the norm in the nonlocal functional produces anisotropic effects in the singular limit. 
The norm 
$$
v\mapsto\|v\|_{Z^*_pK},
$$ 
can be explicitly written and, in the particular case $\|\cdot\|_{K}=|\cdot|$ (Euclidean case), then $K=B_1$, the unit ball of $\R^N$,  and the results are consistent with classical formulas, since $\|\cdot\|_{Z^*_pB}=\sqrt[p]{K_{p,N}}|\cdot|$.

M. Ludwig's proof of formula \eqref{ludw} relies on a reduction argument involving the one dimensional version of the Bourgain-Brezis-Mironescu formula in the Euclidean setting jointly with the  {\em Blaschke-Petkantschin} geometric integration formula  (cf.\ \cite[Theorem 7.2.7]{book}), namely
\begin{equation*}
\int_{\R^N}\int_{\R^N} f(x,y) dxdy=\int_{{\rm Aff(N,1)}}\int_{L}\int_{L} f|_{L\times L}(x,y)
|x-y|^{N-1}d{\mathscr H}^1(x)d{\mathscr H}^1(y)dL,
\end{equation*}
where ${\mathscr H}^1$ is the one dimensional Hausdorff measure on $\R^N$, ${\rm Aff}(N,1)$
is the affine Grassmannian of lines in $\R^N$ and $dL$ denotes the integration with respect to a Haar measure on ${\rm Aff}(N,1)$. 

Around 2006, motivated by an estimate for the topological degree raising in the framework of Ginzburg-Landau equations \cite{BBNg1},  a new alternative characterization of the Sobolev spaces was introduced (cf. \cite{BourNg, nguyen06,NgSob2}). 
As a result, for every $u\in W^{1,p}(\R^N)$ with $p>1$, there holds
\begin{equation}
\label{nonconvex}
\lim_{\delta\to 0}\iint_{\{|u(y)-u(x)|>\delta\}}\frac{\delta^p}{|x-y|^{N+p}}dx dy= 
K_{p,N}\int_{\R^N}|\nabla u|^pdx,
\end{equation}
where $K_{p,N}$ is the constant appearing in \eqref{const}.
It is thus natural to wonder if, replacing $|x-y|$ in the singular kernel with the corresponding 
anisotropic version $\|x-y\|_K$, produces in the limit the same result as in 
formula \eqref{ludw}. 

The previous two characterizations were also considered for $p=1$. $BV$ functions are involved in this case, see \cite{bourg, Davila, BourNg, nguyen06}. Other properties related to these characterizations can be found in  \cite{BHN3, BHN2, BHN, ponce, NgGamma, Ng11, magn-case}. 
Both the  characterizations (for the Euclidean norm) were recently extended to the case of {\em magnetic Sobolev} and  {\em $BV$ spaces}  \cite{acv-p,magn-case,squ-volz}. More general nonlocal functionals have been investigated in  \cite{BHN3,BHN2,BHN-0,BHN,BrezisNguyen}.

\subsection{Anisotropic spaces} In this section, we introduce anisotropic magnetic Sobolev and BV spaces. For this end, complex numbers and notations are involved. 
Let $p\geq 1$ and consider the complex space 
$(\C^N, |\cdot|_{p})$ endowed with
\begin{equation*}
|z|_p:=\left(|(\Re z_1,\ldots, \Re z_N)|^p+|(\Im z_1,\ldots, \Im z_N)|^p\right)^{1/p},
\end{equation*}
where $\Re a$ and  $\Im a$ denote the real and imaginary parts of $a\in\mathbb{C}$. 
Recall that  $| x |$ is the {\em Euclidean} norm of $x \in \R^N$.  Notice that $|z|_p=|z|$ for $z \in \R^N$. Let $\|\cdot\|_K$ be the norm
as in \eqref{Knorm}. We set 
\begin{equation}\label{Lpmoment}
\|v\|_{Z^*_pK}:=\left(\frac{N+p}{p}\int_K |v\cdot x|^p_p dx\right)^{1/p}, \quad  \mbox{ for } v\in\C^N.
\end{equation}
The set $Z^*_pK\subset \C^N$ which is defined as 
\begin{equation}\label{Zpk}
Z^*_pK:=\big\{v\in\C^N:\, \|v\|_{Z^*_pK}\leq 1\big\}
\end{equation}
is called the (complex) polar $L_p$-moment body of $K$. 
Denote $L^p(\R^N,\C)$ 
the Lebesgue space of functions $u:\R^N\to\C$ such that
$$
\|u\|_{L^p(\R^N)}:=\left(\int_{\R^N} |u|_{p}^pdx\right)^{1/p}<\infty. 
$$
For a locally bounded function $A:\R^N\to\R^N$ (magnetic potential), set 
$$
[u]_{W^{1,p}_{A,K}(\R^N)}:=\left(\int_{\R^N}\|\nabla u-\i A(x)u\|^p_{Z^*_pK}dx\right)^{1/p}. 
$$
Let $W^{1,p}_{A,K}(\R^N)$ be the space of $u\in L^p(\R^N,\C)$ such that  $[u]_{W^{1,p}_{A,K}(\R^N)}<\infty$ with the norm
$$
\|u\|_{W^{1,p}_{A,K}(\R^N)}:=\left(\|u\|_{L^p(\R^N)}^p+[u]_{W^{1,p}_{A,K}
	(\R^N)}^p\right)^{1/p}.
$$
Denote $\| \cdot  \|_{Z_1^*K^*}$ the dual norm of the norm $\| \cdot  \|_{Z_1^*K}$ on $\R^N$, namely for $v\in\R^N$
$$
\|v\|_{Z_1^*K^*}:=\sup\big\{\langle v,w\rangle_{\R^N}: w\in\R^N,\,\, \|w\|_{Z_1^*K}\leq 1\big\},\,\,\, \mbox{ with }
\langle v,w\rangle_{\R^N}=\sum_{j=1}^N v_jw_j, \;  \forall  v, w \in \R^N.
$$  
For a complex function $ u \in L^1_{\loc}(\R^N)$, as in \cite{acv-p},  we define 
$$
|Du |_{A, K} := C_{1, A, K, u} + C_{2, A, K, u},
$$ 
where 
\begin{align*}
& C_{1, A, K, u} := \sup \left\{ \int_{\R^N} \Re u \mbox{div} \varphi - A \cdot \varphi \Im u \, dx,\,\, \varphi \in C^1_{c}(\R^N, \R^N) \mbox{ with }  \|\varphi (x)\|_{Z_1^*K^*}  \le 1 \mbox{ in } \R^N \right\}, \\
& C_{2, A, K, u} := \sup \left\{ \int_{\R^N} \Im u \mbox{div} \varphi + A \cdot \varphi \R u \, dx,\,\, \varphi \in C^1_{c}(\R^N, \R^N) \mbox{ with }  \|\varphi (x)\|_{Z_1^*K^*}  \le 1 \mbox{ in } \R^N \right\}.
\end{align*}
We say that $u \in BV_{A,K}(\R^N)$ if $u\in L^1(\R^N)$ and $|Du|_{A,K}<\infty$
and in this case we formally set 
\begin{equation}
\label{agreem}
|Du |_{A, K}=\int_{\R^N} \|\nabla u - \i A(x) u\|_{Z_1^*K} \, dx.
\end{equation}
The space $BV_{A,K}(\R^N)$ is a Banach space \cite{acv-p} equipped the norm
$$
\|u \|_{A, K}=\|u\|_{L^1(\R^N)}+|Du |_{A, K},\quad u\in BV_{A,K}(\R^N).
$$

\subsection{Main results}
The goal of this paper is to extend the two characterizations mentioned above to anisotropic magnetic 
Sobolev and BV spaces. Our approach is in the spirit of the works on the Euclidean spaces. 
In particular, we make no use of the Blaschke-Petkantschin geometric integration formula as in the work of M. Ludwig. 

	Let $A: \R^N \to \R^N$ be measurable and locally bounded. Set 
$$
\Psi_u(x,y):=e^{\i (x-y)\cdot A\left(\frac{x+y}{2}\right)} u (y),\quad x,y\in\R^{N}.
$$ 	
Motivated by the study of the interaction of particles in the presence of a magnetic field, see e.g., \cite{AHS, I10} and references therein, Ichinose \cite{I10} considered the non-local functional 
\begin{equation*}
	H^s_A(\R^N)\ni u\mapsto \iint_{\R^{2N}}\frac{|u(x)-e^{\i (x-y)\cdot A\left(\frac{x+y}{2}\right)}u(y)|^2}{|x-y|^{N+2s}}dx \, dy,
\end{equation*}
for $s\in (0,1)$, and established that its gradient is  the fractional Laplacian associated with the magnetic field $A$ via a probabilistic argument. 
As in the spirit of the previous results, the quantity $\Psi_u$ has been recently involved in the characterization of magnetic Sobolev and BV functions. In this paper, we establish  the following  anisotropic magnetic version of \eqref{nonconvex}. 

\begin{theorem}
	\label{thm2-stat}
	Let $p>1$ and let $A: \R^N \to \R^N$ be Lipschitz. Then, for every $u\in W^{1,p}_{A,K}(\R^N)$,
	\begin{equation*}
	\lim_{\delta\searrow 0} \int_{\R^N}\int_{\R^N} 
	\frac{\delta^p}{\|x-y\|_K^{N+p}}\1_{\{|\Psi_u(x,y)-\Psi_u(x,x)|_p>\delta\}}
	\, dxdy =\int_{\R^N} \|\nabla u-\i A(x)u\|^p_{Z^*_pK}dx,
	\end{equation*}
\end{theorem}


If $p = 1$, one can show  (see Remark~\ref{rem-W11}) that, for $u \in W^{1, 1}_{A, K}(\R^N)$, 
\begin{equation*}
	\lim_{\delta\searrow 0} \int_{\R^N}\int_{\R^N} 
	\frac{\delta}{\|x-y\|_K^{N+1}}\1_{\{|\Psi_u(x,y)-\Psi_u(x,x)|_p>\delta\}}
	\, dxdy  \ge \int_{\R^N} \|\nabla u-\i A(x)u\|_{Z^*_1K}dx.
\end{equation*}
Nevertheless, such an inequality does not hold in general for $u \in BV_{A}(\R^N)$  even in  the case where $A \equiv 0$ and $K $ is  the unit ball (see \cite[Pathology 3]{BHN}). In the case $A = 0$, one has
\begin{equation*}
	\lim_{\delta\searrow 0} \int_{\R^N}\int_{\R^N} 
	\frac{\delta}{\|x-y\|_K^{N+1}}\1_{|u(y)- u(x)|>\delta\}}
	\, dxdy  \ge C  \int_{\R^N} \|\nabla u\|_{Z^*_1K}dx,
\end{equation*}
for some positive constant $0 <  C < K_{N, 1}$.  This inequality is a direct consequence of the corresponding result in the Euclidean setting in \cite{BourNg}. 

\medskip 
We next discuss the BBM formula for the anisotropic magnetic setting. 
Let $(\rho_n)$ be a sequence of non-negative radial mollifiers such that 
\begin{equation}\label{mollifiers-stat}
\lim_{n \to + \infty} \int_{\delta}^\infty \rho_n(r) r^{N-1 -p} \, d r = 0, \quad \mbox{ for all } \delta > 0 \quad \mbox{ and } \quad \int_{0}^1\rho_n(r) r^{N-1} \, d r = 1. 
\end{equation}
Here is the anisotropic magnetic BBM formula. 
\begin{theorem} 
	\label{thm3-stat} 
	Let  $p \geq 1$, let $A:\R^N\to\R^N$ be Lipschitz,  and let $\{\rho_n\}_{n\in\N}$ be a sequence of nonnegative radial mollifiers satisfying \eqref{mollifiers-stat}. 
	Then, for  $u \in W^{1,p}_{A,K}(\R^N)$, 
	\begin{equation*}
	\lim_{n \to + \infty} \iint_{\R^{2N}} \frac{|\Psi_u (x, y) -  \Psi_u(x, x)|^p}{\|x - y\|_K^p} \rho_n ( \|x - y\|_K) \, dx \, dy =  p\int_{\R^N} \|\nabla u - \i A(x) u\|_{Z_p^*K}^p \, dx. 
	\end{equation*}
	Furthermore, if $p=1$ and $u\in BV_{A,K}(\R^N)$ the formula holds with the 
agreement \eqref{agreem}.
\end{theorem}

\begin{remark} \rm Let $(s_n)$ be a positive sequence converging to 0 and set, for $n \ge 1$,  
$$
\rho_n(r)=\frac{p(1-s_n)}{r^{N+ps_n-p}},\quad r>0. 
$$
Then $(\rho_n)$ satisfy \eqref{mollifiers-stat}. Applying Theorem~\ref{thm3-stat}, one rediscovers the results of M. Ludwig. 
\end{remark}

\begin{remark}\rm 
Theorems \ref{thm2-stat} and \ref{thm3-stat} provide the full solution
of a problem arised by Giuseppe Mingione on September 21th, 2016, 
at the end of the seminar ``{\em Another triumph for De Giorgi's Gamma convergence}'' by Haim Brezis at the conference 
``{\em A Mathematical tribute to Ennio De Giorgi}'', held in Pisa from 19th to 23th September 2016.
	\end{remark}

The above results provide an extension of \cite{bourg,BourNg,Davila,ludwig,ludwig-BV,magn-case,nguyen06,acv-p,ponce,squ-volz} to the anisotropic case. 


\section{Proof of Theorem~\ref{thm2-stat}}

\noindent
Set, for $p \ge 1$, 
$$
I_\delta^K(u):=\int_{\R^N}\int_{\R^N} 
\frac{\delta^p}{\|x-y\|_K^{N+p}}\1_{\{|\Psi_u(x,y)-\Psi_u(x,x)|_p>\delta\}}  \, dx \, dy  ,\quad \mbox{ for } u \in L^1_{\loc}(\R^N).
$$
It is clear that, for $u, v \in W^{1, p}_{A, K}(\R^N)$ and $0 < \eps < 1$, 
\begin{equation}\label{Limit formula-2}
I_\delta^K (u) \le  (1-\varepsilon)^{-p} I_{(1-\varepsilon)\delta}^K (v)+ \varepsilon^{-p}I_{\varepsilon\delta}^K(u-v).
\end{equation}	
Applying  \cite[Theorem 3.1]{magn-case}, we have, for $p>1$ and $u \in W^{1,p}_{A, K}(\R^N)$, 
\begin{equation*}
I_\delta^K (u)
\leq C_{N, p, K}  \left( \int_{\R^N}|\nabla u-\i A(x)u|_p^p\, dx + \big(\|\nabla A\|_{L^\infty(\R^N)}^p + 1 \big) \int_{\R^N} |u|_p^p  \, dx \right),  
	\end{equation*}
for some positive constant $C_{N, p, K}$ depending only on $N$, $p$, and $K$. By the density of $C^1_{c}(\R^N)$ in $W^{1, p}_{A}(\R^N)$, 
	it hence suffices to consider the case $u \in C^1_{c}(\R^N)$ which will be assumed from later on. 
	
	By a change of variables as above, we have 
\begin{multline*}
\int_{\R^N}\int_{\R^N} 
\frac{\delta^p}{\|x-y\|_K^{N+p}}\1_{\{|\Psi_u(x,y)-\Psi_u(x,x)|_p>\delta\}}  \, dx \, dy \\[6pt]
= \int_{\R^N} \int_{\mS^{N-1}} \int_0^\infty \frac{1}{\|\sigma\|_K^{N+p}h^{1+p}}\1_{\{|\Psi_u(x,x + \delta h\sigma)-\Psi_u(x,x)|_p>\delta\}}  \, d h \, d \sigma \, dx. 
\end{multline*}
Using the fact
\begin{equation}\label{limit}
\lim_{\delta \to 0 } \frac{ |\Psi_u(x,x + \delta h \sigma)-\Psi_u(x,x)|_p}{\delta} = |\big( \nabla u - \i A(x) u \big) \cdot \sigma|_p h, 
\end{equation}
as in the proof of \cite[Lemma 3.3]{magn-case}, we obtain 
\begin{equation}\label{thm1-p3}
\lim_{\delta \to 0} \int_{\R^N}\int_{\R^N} 
\frac{\delta^p}{\|x-y\|_K^{N+p}}\1_{\{|\Psi_u(x,y)-\Psi_u(x,x)|_p>\delta\}} =  \frac{1}{p} \int_{\R^N} \int_{\mS^{N-1}} \frac{|(\nabla u - \i A (x) u) \cdot \sigma|^p_p}{\| \sigma\|_K^{N+p}}  \, d \sigma \, dx. 
\end{equation}
Since we have
\begin{equation}
\label{id2}
(N+p) \int_{K} |v \cdot y|^p_p \, dy = 
(N+p) \int_{\mS^{N-1}} \int_0^{1/\| \sigma \|_K}  |v \cdot \sigma|^p_p t^{N-1 + p } \, dt d \sigma =  \int_{\mS^{N-1}} \frac{|v \cdot \sigma|^p_p}{\| \sigma\|_K^{N+p}}  \, d \sigma,
\end{equation}
the assertion follows.

\begin{remark}\label{rem-W11} \rm In the case $u \in W^{1, 1}_A(\R^N)$, by Fatou's lemma, as in \eqref{thm1-p3}, one has, with $p = 1$, 
\begin{equation*}
\lim_{\delta \to 0 } I^K_\delta(u) \ge \mathop{\int_{\R^N} \int_{\mS^{N-1}} \int_0^\infty}_{\{ |(\nabla u(x) - \i A (x) u(x) ) \cdot \sigma|_1 h > 1\}} \frac{1}{\| \sigma\|_K^{N+p} h^{2}} \, d h \, d \sigma \, dx = 
\frac{1}{p} \int_{\R^N} \int_{\mS^{N-1}} \frac{| \big( \nabla u - \i A (x) u \big) \cdot \sigma|_1}{\| \sigma\|_K^{N+1}}  \, d \sigma \, dx. 
\end{equation*}
This implies 
\begin{equation*}
	\lim_{\delta\searrow 0} \int_{\R^N}\int_{\R^N} 
	\frac{\delta}{\|x-y\|_K^{N+1}}\1_{\{|\Psi_u(x,y)-\Psi_u(x,x)|_1>\delta\}}
	\, dxdy  \ge \int_{\R^N} \|\nabla u-\i A(x)u\|_{Z^*_1K}dx. 
\end{equation*}
\end{remark}

\section{Proof of Theorem~\ref{thm3-stat}}
\subsection{Proof of Theorem~\ref{thm3-stat} for $p>1$}
Using \cite[Theorem 2.1]{magn-case} without loss of generality, one might assume that $ u \in C^1_c(\R^N)$. 
Note that 
\begin{multline*}
\iint_{\R^{2N}} \frac{|\Psi_u (x, y) -  \Psi_u(x, x)|_p^p}{\|x - y\|_K^p} \rho_n ( \|x - y\|_K) \, dx \, dy  \\[6pt]
=  \int_{\R^N} \int_{\mS^{N-1}} \int_0^\infty  \frac{|\Psi_u (x, x + h \sigma) -  \Psi_u(x, x)|_p^p}{\|\sigma\|_K^p h^p} \rho_n (\| \sigma\|_K h ) h^{N-1} \, dh \, d \sigma \, dx.  
\end{multline*} 
Using  \eqref{limit}, one then can check that, for $p \ge 1$ and $ u \in C^1_c(\R^N)$, 
\begin{multline*}
\lim_{n \to + \infty} \iint_{\{|x-y|\leq 1\}} \frac{|\Psi_u (x, y) -  \Psi_u(x, x)|_p^p}{\|x - y\|_K^p} \rho_n (\|x - y\|_K) \, dx \, dy \\[5pt] 
=  \int_{\R^N} \int_{\mS^{N-1}} \frac{|(\nabla u - \i A (x) u) \cdot \sigma|_p^p}{\|\sigma \|_K^p} \, d \sigma \, dx  \lim_{n \to + \infty} \int_0^1 \rho_n (\| \sigma\|_K h ) h^{N-1} \, dh.
\end{multline*} 
Furthermore, observe that
\begin{equation*}
\iint_{\{|x-y|>1\}} \frac{|\Psi_u (x, y) -  \Psi_u(x, x)|_p^p}{\|x - y\|_K^p} \rho_n (\|x - y\|_K) \, dx \, dy \leq
C\| u\|_{L^p}^p \int_1^\infty h^{N-1-p} \rho_n(\| \sigma\|_K h) \, dh . 
\end{equation*} 
Therefore, for $p \ge 1$ and $ u \in C^1_c(\R^N)$, on account of \eqref{mollifiers-stat} we obtain
\begin{equation}\label{main-BBM-2}
\lim_{n \to + \infty} \iint_{\R^{2N}} \frac{|\Psi_u (x, y) -  \Psi_u(x, x)|_p^p}{\|x - y\|_K^p} \rho_n (\|x - y\|_K) \, dx \, dy
=  \int_{\R^N} \int_{\mS^{N-1}} \frac{|(\nabla u - \i A (x) u) \cdot \sigma|_p^p}{\|\sigma \|_K^{N+p}} \, d \sigma \, dx. 
\end{equation} 
The conclusion now follows from \eqref{id2}. 

\subsection{Proof of Theorem~\ref{thm3-stat} for $p=1$}
We first present some preliminary results. The first one is the following

\begin{lemma}\label{lem-BV-0} Let $u \in W_{A,K}^{1,1}(\R^N)$. Then
	\begin{equation*}
	|Du |_{A, K}=\int_{\R^N} \|\nabla u - \i A(x) u\|_{Z_1^*K} \, dx.
	\end{equation*}
\end{lemma}
\begin{proof} The proof is quite standard and based on integration by parts after noting that 
$$
\|\nabla u - \i A(x) u\|_{Z_1^*K} = \|\nabla \Re u - A(x) \Im u \|_{Z_1^*K} +   \|\nabla \Im u +  A(x) \Re u \|_{Z_1^*K},
$$
since $A(x) \in \R^N$ for $x \in \R^N$.  The details are left to the reader. 
\end{proof}

\begin{lemma}\label{lem-BV-1} Let $u \in BV_A(\R^N)$ and $(u_n) \subset BV_A(\R^N)$. Assume that 
$$
\lim_{n \to + \infty} u_n =  u \mbox{ in } L^1(\R^N).   
$$
Then 
$$
\liminf_{n \to + \infty} |Du_n |_{A, K} \ge |Du |_{A, K}.  
$$
\end{lemma}

\begin{proof} One can check that 
$$
\liminf_{n \to + \infty} C_{1, A, K, u_n} \ge C_{1, A, K, u} \quad \mbox{ and } \quad \liminf_{n \to + \infty} C_{2, A, K, u_n} \ge C_{2, A, K, u}. 
$$
The conclusion follows. 
\end{proof}

For $r > 0$, let $B_r$ denote the ball centered at the origin and of radius $r$. We have

\begin{lemma}\label{lem-BV-2} Let $u \in BV_A(\R^N)$ and let  $(\tau_m)$ be a sequence of nonnegative mollifiers with $\supp \tau_m \subset B_{1/m}$ which is normalized by the condition $\int_{\R^N} \tau_m(x) \, dx = 1$. Set $u_m = \tau_m * u.$
Assume that $A$ is Lipschitz.  Then 
$$
\lim_{m \to + \infty} |Du_m |_{A, K} = |Du |_{A, K}.  
$$
\end{lemma}

\begin{proof} The proof is quite standard, see  e.g., \cite{gariepy} and also \cite{acv-p}. Let $\varphi \in C^1_{c}(\R^N)$ be such that 
\begin{equation*}
\|\varphi (x)\|_{Z_1^*K^*}  \le 1\,\, \mbox{ in } \R^N. 
\end{equation*}
We have 
\begin{multline}\label{lem-BV-2-p2}
 \int_{\R^N} \Re u_m \mbox{div} \varphi - A \cdot \varphi \Im u_m \, dx  =  \int_{\R^N} \Re u \mbox{div} \varphi_m - A \cdot \varphi_m \Im u \, dx \\[6pt]
+  \int_{\R^N} \int_{\R^N} \big(A(x) - A(x-y) \big) \cdot \varphi(x-y) \tau_m(y) u(x) \, dx \, dy. 
\end{multline}
Since 
$$
\|\varphi_m (x)\|_{Z_1^*K^*} \le \sup_{y} \|\varphi (y)\|_{Z_1^*K^*}   \le 1, 
$$
we have 
\begin{equation}\label{lem-BV-2-p3}
\Big|\int_{\R^N} \Re u \mbox{div} \varphi_m - A \cdot \varphi_m \Im u \, dx\Big| \le C_{1, A, K, u} 
\end{equation}
Since $\supp \tau_m \subset B_{1/m}$, one can check that 
\begin{equation}\label{lem-BV-2-p4}
\Big| \int_{\R^N} \int_{\R^N} \big(A(x) - A(x-y) \big) \cdot \varphi(x-y) \tau_m(y) u(x) \, dx \, dy \Big| \le C \| \nabla A\|_{L^\infty} \| u \|_{L^1}/ m. 
\end{equation}
A combination of \eqref{lem-BV-2-p2}, \eqref{lem-BV-2-p3}, and \eqref{lem-BV-2-p4} yields 
$$
\limsup_{m \to + \infty} C_{1, A, K, u_m} \le C_{1, A, K, u}. 
$$
Similar, we obtain 
$$
\limsup_{m \to + \infty} C_{2, A, K, u_m} \le C_{2, A, K, u}
$$
and the conclusion follows from Lemma~\ref{lem-BV-1}. 
\end{proof}

We are ready to give 

\medskip 
\noindent{\bf Proof of Theorem \ref{thm3-stat} for $p=1$.} Let $(\tau_m)$ be a sequence of nonnegative mollifiers with $\supp \tau_m \subset B_{1/m}$ which is normalized by the condition $\int_{\R^N} \tau_m(x) \, dx = 1$. Set 
$u_m  = u * \tau_m.$
As in the proof of \cite[Lemma 2.4]{magn-case}, we have 
\begin{multline*}
\iint_{\R^{2N}} \frac{|\Psi_{u_m} (x, y) -  \Psi_{u_m}(x, x)|_1}{\|x - y\|_K} \rho_n ( \|x - y\|_K) \, dx \, dy \\[6pt] 
 \le \iint_{\R^{2N}} \frac{|\Psi_u (x, y) -  \Psi_u(x, x)|_1}{\|x - y\|_K} \rho_n ( \|x - y\|_K) \, dx \, dy \\[6pt] + C \int_{\R^{N}}\int_{\R^{N}}\int_{\R^{N}} |z| \tau_m(z)   \rho_n ( \|x - y\|_K) u(y) \, dz \, dx \, dy. 
\end{multline*}
We have 
\begin{multline*}
\lim_{m \to + \infty} \lim_{n \to + \infty} \iint_{\R^{2N}} \frac{|\Psi_{u_m} (x, y) -  \Psi_{u_m}(x, x)|_1}{\|x - y\|_K} \rho_n ( \|x - y\|_K) \, dx \, dy \\[6pt]
\ge \lim_{m \to + \infty} \int_{\R^N} \|\nabla u_m - \i A(x) u_m \|_{Z_1^*K} \, dx \mathop{=}^{\mathrm{Lemma~\ref{lem-BV-2}}}  \int_{\R^N} \|\nabla u - \i A(x) u \|_{Z_1^*K} \, dx
\end{multline*}
and, since $\supp \tau_m \subset B_{1/m}$, 
\begin{equation*}
\int_{\R^N}\int_{\R^N}\int_{\R^N} |z| \tau_m(z) 
 \rho_n ( \|x - y\|_K) u(y) \, dz \, dx \, dy  \le C/ m. 
\end{equation*}
It follows that 
\begin{equation*}
\liminf_{n \to + \infty} \iint_{\R^{2N}} \frac{|\Psi_u (x, y) -  \Psi_u(x, x)|_1}{\|x - y\|_K} \rho_n ( \|x - y\|_K) \, dx \, dy \ge \int_{\R^N} \|\nabla u - \i A(x) u\|_{Z_1^*K} \, dx.
\end{equation*}
We also have, by Fatou's lemma,  
\begin{multline}\label{thm3-1-p00}
\iint_{\R^{2N}} \frac{|\Psi_u (x, y) -  \Psi_u(x, x)|_1}{\|x - y\|_K} \rho_n ( \|x - y\|_K) \, dx \, dy \\[6pt]
\le \liminf_{m \to + \infty}
\iint_{\R^{2N}} \frac{|\Psi_{u_m} (x, y) -  \Psi_{u_m}(x, x)|_1}{\|x - y\|_K} \rho_n ( \|x - y\|_K) \, dx \, dy.  
\end{multline}
We next derive an upper bound for the RHS of \eqref{thm3-1-p00}.
Let $v \in W^{1, 1}_{A,K}(\R^N) \cap C^\infty(\R^N)$. We have 
\begin{multline}\label{thm-3-1-p1}
\iint_{\R^{2N}} \frac{|\Psi_{v} (x, y) -  \Psi_{v}(x, x)|_1}{\|x - y\|_K} \rho_n ( \|x - y\|_K) \, dx \, dy \\[6pt]
= \int_{\R^N} \int_{\mS^{N-1}} \int_0^\infty \frac{|\Psi_{v} (x, x + h \sigma) -  \Psi_{v}(x, x)|_1}{\|h \sigma\|_K} \rho_n ( h \| \sigma \|_K) h^{N-1} \, d h \, d \sigma \, dx. 
\end{multline}
Using the fact
\begin{align*}
\frac{\partial \Psi_v (x, y)}{\partial y}  &= e^{\i (x - y)\cdot A \left( \frac{x + y}{2} \right)} \nabla v(y)- \i \left\{ A\Big( \frac{x + y}{2} \Big) + \frac{1}{2} (y-x) \cdot \nabla A \Big(\frac{x + y}{2} \Big)\right\} \times \\
& \times e^{\i (x - y)\cdot A \left( \frac{x + y}{2} \right)} v(y),  \notag
\end{align*}
and applying the mean value theorem, we obtain 
\begin{align}\label{thm-3-1-p2}
\iint_{\R^{2N}} & \frac{|\Psi_{v} (x, y)  -  \Psi_{v}(x, x)|_1}{\|x - y\|_K}  \rho_n ( \|x - y\|_K) \, dx \, dy \\[6pt]
\le &  \int_{\R^N} \int_{\mS^{N-1}} \int_0^1 \int_0^1 | \big( \nabla v - \i A v \big) \cdot \sigma|_1 (x + t h \sigma) \frac{1}{\|\sigma \|_K} h^{N-1} \rho_n(h \|\sigma\|_K)\, dt \, dh \, d \sigma \, dx  \nonumber \\[6pt]
& + C \int_{\R^N} \int_{\mS^{N-1}} \int_0^1 \int_0^1  |v(x + t h \sigma )| \frac{\|\nabla A \|_{L^\infty} h}{\|\sigma \|_K} h^{N-1} \rho_n(h \|\sigma\|_K)\, dt \, d h \, d \sigma \, dx \nonumber \\[6pt]
& + C \mathop{\int_{\R^N} \int_{\R^N}}_{\{|x - y| > 1\}} \big( |v(x)| + |v(y)| \big) \frac{1}{\|x-y\|_K}\rho_n(\|x-y\|_K) \,d x \, dy.  \nonumber
\end{align}
One can check that
\begin{multline}\label{thm-3-1-p3}
\int_{\R^N} \int_{\mS^{N-1}} \int_0^1 \int_0^1 |\big(\nabla v - \i A v \big) \cdot \sigma|_1 \big(x + t h \sigma \big) \frac{1}{\|\sigma \|_K} h^{N-1} \rho_n(h \|\sigma\|_K)\, dt \, dh \, d \sigma \, dx \\[6pt]
\leq  \int_{\R^N} \int_{\mS^{N-1}} \frac{|(\nabla v - \i A (x) v) \cdot \sigma|_1}{\|\sigma \|_K^{N+1}} \, d \sigma \, dx \int_0^\lambda h^{N-1} \rho_n(h) \, dh,
\end{multline}
and 
\begin{multline}\label{thm-3-1-p4}
\int_{\R^N} \int_{\mS^{N-1}}  \int_0^1 \int_0^1  |v(x + t h \sigma \big)| \frac{\|\nabla A \|_{L^\infty} h}{\|\sigma \|_K} h^{N-1} \rho_n(h \|\sigma\|_K)\, dt  \, d h \, d \sigma \, dx \\[6pt]
\le C_K \|\nabla A \|_{L^\infty} \| v\|_{L^1} \int_0^\lambda h^{N} \rho_n(h) \, dh, 
\end{multline}
where $\lambda = \max\{ \| \sigma\|_K : \sigma \in \mS^{N-1}\}$. 
A combination of \eqref{thm-3-1-p1}, \eqref{thm-3-1-p2}, \eqref{thm-3-1-p3}, and \eqref{thm-3-1-p4} yields
\begin{multline}\label{thm-3-1-p5}
\iint_{\R^{2N}} \frac{|\Psi_{v} (x, y) -  \Psi_{v}(x, x)|_1}{\|x - y\|_K} \rho_n ( \|x - y\|_K) \, dx \, dy 
\le 
\int_0^\lambda h^{N-1} \rho_n(h) \, dh\int_{\R^N} \|\nabla v - \i A(x) v\|_{Z_1^*K} \, dx \\[6pt]  
+ C_K (\|\nabla A \|_{L^\infty} + 1)  \| v\|_{L^1} \Big( \int_0^\lambda h^{N} \rho_n(h) \, dh + \int_1^\infty h^{N-2} \rho_n(h) \, dh \Big). 
\end{multline}
Using Lemma~\ref{lem-BV-2}, we derive from \eqref{thm3-1-p00} and \eqref{thm-3-1-p5} that 
\begin{multline*}
\iint_{\R^{2N}} \frac{|\Psi_u (x, y) -  \Psi_u(x, x)|_1}{\|x - y\|_K} \rho_n ( \|x - y\|_K) \, dx \, dy \\[6pt]
\le \int_0^\lambda h^{N-1} \rho_n(h)\int_{\R^N} \|\nabla u - \i A(x) u\|_{Z_1^*K} \, dx \\[6pt]
+ C_K \|\nabla A \|_{L^\infty} \| u\|_{L^1} \Big( \int_0^\lambda h^{N} \rho_n(h) \, dh + \int_1^\infty h^{N-2} \rho_n(h) \, dh \Big), 
\end{multline*}
which yields, by \eqref{mollifiers-stat}, 
\begin{equation*}
\limsup_{n \to + \infty} \iint_{\R^{2N}} \frac{|\Psi_u (x, y) -  \Psi_u(x, x)|_1}{\|x - y\|_K} \rho_n ( \|x - y\|_K) \, dx \, dy 
\le \int_{\R^N} \|\nabla v - \i A(x) v\|_{Z_1^*K} \, dx.  
\end{equation*}
The proof is complete. \qed

%

\bigskip

\end{document}